\documentclass[reqno]{amsart}

\usepackage[utf8]{inputenc}
\usepackage[T1]{fontenc}
\usepackage[english]{babel}
\usepackage{hyphenat}	

\usepackage{float}
\usepackage[unicode=true, pdfusetitle, bookmarks=true,bookmarksnumbered=false,bookmarksopen=false,
 breaklinks=false,pdfborder={0 0 1},backref=false,colorlinks=false]{hyperref} 

\usepackage{amsfonts}
\usepackage{latexsym}
\usepackage{amssymb}
\usepackage{amsmath}
\usepackage{graphicx}  

\usepackage{mathrsfs}
\usepackage{subcaption} 

\usepackage{enumitem}

\usepackage{xcolor}
\usepackage{hyperref}

\numberwithin{equation}{section}

\theoremstyle{definition} 
\newtheorem{df}{Definition}

\newtheorem{rem}[df]{Remark}
\theoremstyle{plain}
\newtheorem{thm}[df]{Theorem} 

\newtheorem{prop}[df]{Proposition}
\newtheorem{lemma}[df]{Lemma}

\makeatletter
\@namedef{subjclassname@2020}{%
	\textup{2020} Mathematics Subject Classification}
\makeatother

\DeclareMathOperator{\diam}{diam}
\DeclareMathOperator{\dist}{dist}
\DeclareMathOperator{\DET}{DET}
\DeclareMathOperator{\rdet}{rdet}
\DeclareMathOperator{\Per}{Per}
\DeclareMathOperator{\EPer}{EPer}
\DeclareMathOperator{\Orb}{Orb}

\newcommand{\mmod}{\operatorname{mod}}

\begin{document}

\hyphenation{cardi-nali-ty gene-ra-lize asym-pto-tic asser-tion quanti-fi-cation techni-ques gene-ra-lized conti-nuous con-figura-tion}


\title{Correlation sum and recurrence determinism for~interval maps}
\author{Michaela Mihokov\'{a}}

\address{Department of Mathematics, Faculty of Natural Sciences, Matej Bel University,
	Tajovsk\'{e}ho 40, 974 01 Bansk\'{a} Bystrica, Slovakia}

\email{michaela.mihokova@umb.sk}






\begin{abstract}
	Recurrence quantification analysis is a method for measuring the complexity of dynamical systems. Recurrence determinism is a fundamental characteristic of it, closely related to correlation sum. In this paper, we study asymptotic behavior of these quantities for interval maps. We show for which cases the asymptotic correlation sum exists. An example of an interval map with zero entropy and a point with the finite $\omega$-limit set for which the asymptotic correlation sum does not exist is given. Moreover, we present formulas for computation of the asymptotic correlation sum with respect to the cardinality of the $\omega$-limit set or to the configuration of the intervals forming it, respectively. We also show that for a not Li-Yorke chaotic (and hence zero entropy) interval map, the limit of recurrence determinism as distance threshold converges to zero can be strictly smaller than one.
\end{abstract}

\maketitle

\thispagestyle{empty}

\section{Introduction}
A \emph{(discrete) dynamical system} is an ordered pair $(X,f)$ where $X$ is a compact metric space and $f\colon X\to X$ is a continuous -- not necessarily invertible -- map. For predictability of a trajectory of a dynamical system, one may use mathematical tools of recurrence quantification analysis (RQA) \cite{ZW92}. These methods quantify the number and duration of recurrences in dynamical systems, and are visually represented via recurrence plots \cite{EKR87}. The percentage of recurrence points in such plot is called recurrence rate, and it corresponds to the quantity named correlation sum. RQA and recurrence plots have wide applications in medicine, economy, biology, and many other areas; see, e.g., \cite{CVFGZ04, FA05, VR21}, and for a comprehensive overview of the subject, see \cite{MRTK07, WM15}. In this paper, we are focused on correlation sum and~recurrence determinism, one of fundamental characteristics of RQA.

\subsection{Correlation sum}
Recall that if $(X,f)$ is a dynamical system, $\varrho$ is the metric on~$X$ and $m\in\mathbb{N}$ (where $\mathbb{N}$ denotes the set of all positive integers), then \emph{Bowen's metric} $\varrho_m = \varrho^f_m$ on $X$ is defined by
$$\varrho_m(x,y) = \max\left\{\varrho\left(f^i(x), f^i(y)\right)\colon 0\leq i<m\right\}$$
for all $x,y\in X$. For $x\in X$, $\varepsilon>0$ and $n\in\mathbb{N}$, the \emph{correlation sum} of (the beginning of) the trajectory of a point $x$ (with respect to~$\varrho_m$) is 
\begin{equation*}
C_m(x, n, \varepsilon) =\frac1{n^2}\#\left\{(i,j)\colon\ 0\leq i,j<n,\quad \varrho_m\left(f^i(x),f^j(x)\right)\leq\varepsilon\right\}.
\end{equation*}
This quantity is the relative frequency of recurrences occurring in the initial segment of the trajectory of $x$ with respect to the closeness defined by the threshold $\varepsilon$ and the metric $\varrho_m$. The \emph{lower} and \emph{upper asymptotic correlation sums} (with respect to~$\varrho_m$) are
\begin{equation*}
\underline{c}_m(x,\varepsilon)=\liminf\limits_{n\to\infty}C_m(x,n,\varepsilon),\hspace{0.5cm} \hspace{0.5cm}
\overline{c}_m(x,\varepsilon)= \limsup\limits_{n\to\infty}C_m(x,n,\varepsilon).
\end{equation*}
If $\underline{c}_m(x,\varepsilon)=\overline{c}_m(x,\varepsilon)$, i.e., if $\lim_{n\to\infty}C_m(x,n,\varepsilon)$ exists, we say that the \emph{asymptotic correlation sum} exists and we denote this limit by $c_m(x,\varepsilon)$.

In~\cite{Pes93} (see also \cite[Theorem~17.1]{Pes97}), the author proved that if $\mu$ is an $f$-ergodic measure, then correlation sums of $\mu$-almost every point $x\in X$ converges (as $n$ goes to $\infty$) to the \emph{correlation integral}
\begin{equation*}
	c_m\left(\mu,\varepsilon\right) = \mu\times\mu\left\{(y,z)\in X\times X\colon \varrho_m(y,z)\leq\varepsilon\right\}
\end{equation*}
for all but countably many $\varepsilon>0$. As a consequence we have that $\underline{c}_m(x,\varepsilon) = \overline{c}_m(x,\varepsilon)$ for $\mu$-almost every point $x\in X$ and for all but countably many $\varepsilon>0$. The~present paper deals with the question whether the previous equality holds also for \emph{every} $x\in X$ with $X=I$ being a compact unit interval. A similar question was studied in~\cite{Spi19}, however, there it was done for $m=\infty$. Our results are stated with respect to~$\omega$-limit sets $\omega_f(x)$ and are summarized in the following theorem (for the definition of solenoidal set, see Section~\ref{S:solenoidal}). Here, by an \emph{order-preserving metric} on a real interval we mean a metric $\varrho$ such that $\max\left\{\varrho(x,y), \varrho(y,z)\right\}< \varrho\{x,z\}$ for all distinct points $x,y,z\in X$ with $x<y<z$.

\setcounter{df}{0}
\renewcommand{\thedf}{\Alph{df}}
\begin{thm}\label{Thm:main_c}
	Let $I$ be equipped with an order-preserving metric compatible with the Euclidean topology. Let $(I,f)$ be a dynamical system and $x\in I$. Then $$\underline{c}_m(x,\varepsilon) = \overline{c}_m(x,\varepsilon)$$
	\begin{enumerate}
		\item\label{it:Main_c:solenoidal} for every $m\in\mathbb{N}$ and $\varepsilon>0$ if $\omega_f(x)$ is solenoidal; 
		\item\label{it:Main_c:finite} for every $m\in\mathbb{N}$ and all but finitely many $\varepsilon>0$ if $\omega_f(x)$ is finite.
	\end{enumerate}
\end{thm}
\renewcommand{\thedf}{\arabic{df}}
\setcounter{df}{0}

Moreover, we give formulas for computation of these asymptotic correlation sums with respect to the configuration of the intervals forming $\omega_f(x)$ or to the cardinality of $\omega_f(x)$, respectively (see Theorems~\ref{Thm:solenoidal} and~\ref{Thm:finite_c}).
Naturally, one can ask whether result in~Theorem~\ref{Thm:main_c}(\emph{\ref{it:Main_c:finite}}) may be strengthened to hold for \emph{every} $\varepsilon>0$. Proposition~\ref{Prop:c} shows that it cannot be done. Since any $\omega$-limit set of a zero entropy map is either finite or solenoidal (see \cite[Theorems~5.4 and~4.1(d)]{Blo95} and \cite{Smi86}), Theorem~\ref{Thm:main_c} solves all possibilities for dynamical systems $(I,f)$ with zero topological entropy.

\subsection{Recurrence determinism}
We also study recurrence determinism. Recall that the \emph{recurrence $m$-determinism} is
\begin{equation*}
\rdet_m(x,n,\varepsilon)=\frac{C_m(x,n,\varepsilon)}{C_1(x,n,\varepsilon)},
\end{equation*}
and the \emph{lower} and \emph{upper asymptotic recurrence $m$-determinisms} are
\begin{equation*}
\underline{\rdet}_m(x,\varepsilon)=\liminf\limits_{n\to\infty}\rdet_m(x,n,\varepsilon), \hspace{0.5cm}
\hspace{0.5cm}
\overline{\rdet}_m(x,\varepsilon)=\limsup\limits_{n\to\infty}\rdet_m(x,n,\varepsilon),
\end{equation*}
respectively.
If $\underline{\rdet}_m(x,\varepsilon)=\overline{\rdet}_m(x,\varepsilon)$, i.e., if $\lim\nolimits_{n\to\infty}\rdet_m(x,n,\varepsilon)$ exists, we say that the \emph{asymptotic recurrence $m$-determinism} exists and we denote this limit by~$\rdet_m(x,\varepsilon)$. Note that this definition is slightly different from the one commonly used in RQA. However, they are closely related since, by~\cite{GMS13} (see also~\cite{SSMA15}), RQA-determinism is
\begin{equation}\label{eq:Intro_det}
\DET_m(x,n,\varepsilon) = m\cdot\rdet_m(x,n,\varepsilon) - (m-1)\cdot\rdet_{m+1}(x,n,\varepsilon).
\end{equation}
The present paper provides the following result concerning an asymptotic recurrence determinism. Note that a similar result was published in~\cite[Lemma~4.2]{Maj16} and also in~\cite[Proposition~4.1]{Spi19}. 
\setcounter{df}{1}
\renewcommand{\thedf}{\Alph{df}}
\begin{thm}\label{Thm:main_rdet}
	Let $(I,f)$ be a dynamical system such that $x\in I$ and $\omega_f(x)$ is finite. Then, for every $m\in\mathbb{N}$,
	$$\lim\limits_{\varepsilon\to 0}\rdet_m(x,\varepsilon) = 1.$$ 
\end{thm}
\renewcommand{\thedf}{\arabic{df}}
\setcounter{df}{0}
This theorem says that trajectories of the considered systems are perfectly predictable in the -- arbitrarily large -- finite horizon for points with finite $\omega$-limit sets. Therefore, it is true for \emph{every} point of a \emph{strongly non-chaotic} dynamical system, i.e., a system whose every $\omega$-limit set is finite \cite[page~126]{BC92}.
One may try to~generalize this result. In~\cite[Theorem~4.14]{Maj16}, it was proved that $\rdet_m(x,\varepsilon) = 1$ for all $x\in I$, $m\in\mathbb{N}$ and sufficiently small $\varepsilon>0$ if $f\colon I\to I$ belongs to the class of so-called Delahaye maps~\cite{Del80} (see also~\cite[Example~5.56]{Rue17}). Since strongly non-chaotic or Delahaye maps are not Li-Yorke chaotic (and hence their topological entropy is zero), a~natural question arises whether a similar result holds for \emph{every} continuous map that is not chaotic in Li-Yorke sense. In~\cite[Theorem~4.12]{Maj16}, the~author showed that there are a Delahaye map $f\colon I\to I$ and a point $x\in I$ such that $\limsup_{\varepsilon\to 0}\rdet_\infty(x,\varepsilon) < 1$. In the present paper we construct a continuous map $f\colon I\to I$ that is not Li-Yorke chaotic (and so with zero entropy) and a point $x\in I$ such that $\liminf_{\varepsilon\to0}\rdet_m(x,\varepsilon)<1$ for every $m\in\mathbb{N}\setminus\{1\}$, see Proposition~\ref{Prop:rdet}. Therefore, the result from~Theorem~\ref{Thm:main_rdet} cannot be generalized to solenoidal $\omega$-limit sets.

The paper is organized as follows. In~Section~\ref{S:preliminaries}, we present notation and facts that will be required later. Section~\ref{S:solenoidal} addresses solenoidal sets and contains the~proof of~Theorem~\ref{Thm:main_c}(\emph{\ref{it:Main_c:solenoidal}}). Proof of Theorem~\ref{Thm:main_c}(\emph{\ref{it:Main_c:finite}}) is given in~Section~\ref{S:finite}. This section also provides an example of a dynamical system and a point with finite $\omega$-limit set for~which the asymptotic correlation sum does not exist for a specific choice of~$\varepsilon>0$. The results for recurrence determinism are presented in~Section~\ref{S:determinism}.

\section{Preliminaries}\label{S:preliminaries}
We denote the set of all positive integers by $\mathbb{N}$. The set of all nonnegative integers is denoted by $\mathbb{N}_0$, and the set of all real numbers is denoted by $\mathbb{R}$. For $p\in\mathbb{N}$, we sometimes use the symbol $\mathbb{Z}_p$ to denote the set of all nonnegative integers less than $p$, i.e., $\mathbb{Z}_p=\{0,1,\dots, p-1\}$. Symbols $(a,b)$ and $[a,b]$ are used for the open interval and for the closed interval in $\mathbb{R}$ with endpoints $a,b\in\mathbb{R}$, respectively. We denote by $I$ the closed unit interval $[0,1]$. The \emph{intersection} and the \emph{union} of sets $A,B$ is denoted by $A\cap B$ and $A\cup B$, respectively; if $A,B$ are disjoint, then their union is sometimes denoted by $A\sqcup B$. The cardinality of $A$ is denoted by $\#A$.

Let $(X,f)$ be a dynamical system. We say that a point $x\in X$ is \emph{periodic} if $f^h(x)=x$ for some $h\in\mathbb{N}$. The smallest $h\in\mathbb{N}$ satisfying the above condition is called the~\emph{period} of~$x$. If the period of $x$ is one, then $x$ is said to be a \emph{fixed point}. We denote the set of all periodic points of $(X,f)$ by $\Per(f)$. We say that a point $x\in X$ is \emph{eventually periodic} if $f^k(x)$ is a periodic point for some $k\in\mathbb{N}$. The~set of all eventually periodic points of $(X,f)$ is denoted by $\EPer(f)$. Trivially, $\Per(f)\subseteq \EPer(f)$. The \emph{orbit} of $x\in X$ is the set $\Orb_f(x) = \left\{f^n(x)\colon n\in\mathbb{N}_0\right\}$. The symbol $\omega_f(x)$ is used to denote the \emph{$\omega$-limit set} of $x\in X$, i.e., the set of all limit points of the trajectory $\left(f^n(x)\right)_{n\in\mathbb{N}_0}$. A compact interval $J\subseteq X$ is \emph{$p$-periodic} ($p\in\mathbb{N}$) if $f^p(J)=J$ and $f^k(J)$, $0\leq k<p$, are pairwise disjoint.

\subsection{Correlation sum and recurrence determinism}
It is easy to see that (asymptotic) correlation sums are:
\begin{itemize}
	\item numbers from the unit interval $[0,1]$; in~particular, (asymptotic) correlation sums are equal to $1$ for~every $\varepsilon\geq\diam(X)$, where $\diam(X)$ denotes the~diameter of $X$;
	\item nondecreasing functions of $\varepsilon$;
	\item nonincreasing functions of $m$;
\end{itemize}
and (asymptotic) recurrence determinisms are:
\begin{itemize}
	\item numbers from the unit interval $[0,1]$; in~particular, (asymptotic) recurrence determinisms are equal to $1$ for~every $\varepsilon\geq\diam(X)$, or for $m=1$;
	\item nonincreasing functions of $m$.
\end{itemize}

The following is trivial (see, e.g., Lemma~3.4 from~\cite{Spi19}). It says that the values of asymptotic correlation sums and asymptotic recurrence determinisms do not depend on the first finitely many iterates.

\begin{lemma}\label{L:iteration}
	Let $(X,f)$ be a dynamical system, $x\in X$, $h\in\mathbb{N}$, $m\in\mathbb{N}$, and $\varepsilon>0$. Then,
	\begin{equation*}
	\underline{c}_m\left(f^h(x),\varepsilon\right) = \underline{c}_m(x,\varepsilon),
	\hspace{0.5cm}\hspace{0.5cm}
	\overline{c}_m\left(f^h(x),\varepsilon\right) = \overline{c}_m(x,\varepsilon),
	\end{equation*}
	and
	\begin{equation*}
	\underline{\rdet}_m\left(f^h(x),\varepsilon\right) = \underline{\rdet}_m(x,\varepsilon),
	\hspace{0.5cm}\hspace{0.5cm}
	\overline{\rdet}_m\left(f^h(x),\varepsilon\right) = \overline{\rdet}_m(x,\varepsilon).
	\end{equation*}
\end{lemma}

\subsection{Configurations of compact real intervals}

Let $J$, $K$ and $L$ be compact real intervals. We denote the distance between $J$ and $K$ by $\dist(J,K)$, i.e.,
$$\dist(J,K) = \min\left\{\varrho(x,y)\colon x\in J,\ y\in K\right\}.$$
The maximum distance among points in $J$ and $K$ is $\diam(J\cup K)$:
$$\diam(J\cup K) = \max\left\{\varrho(x,y)\colon x,y\in J\cup K\right\}.$$
If $\max J< \min K$, then we write $J<K$.
We write $J\leq K$ if $J<K$ or $J=K$.

Let $n\in\mathbb{N}$, $\varepsilon>0$ and $J_1<J_2<\dots<J_n$ be compact real intervals. Put
$$\mathcal{I}_n(\varepsilon) = \left\{(a,b)\colon 1\leq a,b\leq n,\quad \dist\left(J_a, J_b\right) <\varepsilon <\diam\left(J_a\cup J_b\right)\right\}.$$
Trivially, $(a,a)\in \mathcal{I}_n(\varepsilon)$ for every $a$ with $\diam\left(J_a\right)>\varepsilon$, and
\begin{equation}\label{eq:ab_ba}
(a,b)\in\mathcal{I}_n(\varepsilon) \quad \iff \quad (b,a)\in\mathcal{I}_n(\varepsilon).
\end{equation}
Till the end of this section we assume that the metric $\varrho$ is order-preserving and compatible with the Euclidean topology on the real line. Immediately, we have the~following lemma.

\begin{lemma}\label{L:amongI}
	Let $n\in\mathbb{N}, \varepsilon>0$ and $J_1<J_2<\dots<J_n$ be compact real intervals.
	If $(a,b)\in\mathcal{I}_n(\varepsilon)$ and $(a,d)\in\mathcal{I}_n(\varepsilon)$ for some $1\leq a\leq b<d\leq n$, then $(a,c)\in\mathcal{I}_n(\varepsilon)$ for every integer $c$ such that $b<c<d$.
	
	Similarly, if $(a,d)\in\mathcal{I}_n(\varepsilon)$ and $(c,d)\in\mathcal{I}_n(\varepsilon)$ for some $1\leq a< c\leq d\leq n$, then $(b,d)\in\mathcal{I}_n(\varepsilon)$ for every integer $b$ such that $a<b<c$.
\end{lemma}

The next lemma describes another relationship among pairs in $\mathcal{I}_n(\varepsilon)$.
\begin{lemma}\label{L:bcNot}
	Let $n\in\mathbb{N}, \varepsilon>0$ and $J_1<J_2<\dots<J_n$ be compact real intervals. Assume that integers $a,b,c,d$ are such that $1\leq a<b\leq c<d\leq n$ and $(a,d)\in\mathcal{I}_n(\varepsilon)$. Then $(b,c)\not\in\mathcal{I}_n(\varepsilon)$.
\end{lemma}

\begin{proof}
	For $b=c$ the assertion is trivial, so we may assume that $b<c$. The~configuration $J_a<J_b<J_c<J_d$ implies that $\diam\left(J_b\cup J_c\right) <\dist\left(J_a, J_d\right)$. Hence, using the fact that $(a,d)\in\mathcal{I}_n(\varepsilon)$, $\diam\left(J_b\cup J_c\right) <\varepsilon$, and so $(b,c)\not\in\mathcal{I}_n(\varepsilon)$.
\end{proof}

\begin{rem}
	Obviously, $\#\mathcal{I}_n(\varepsilon)\geq0$ for any $\varepsilon>0$ and $n\in\mathbb{N}$. For any $n\in\mathbb{N}$ and $\varepsilon>0$, it is easy to construct compact real intervals $J_1<J_2<\dots<J_n$ such that $\#\mathcal{I}_n(\varepsilon)=0$. To do this, let $J_1$ be a compact real interval with diameter less than or equal to $\varepsilon$. Then inductively construct the remaining intervals such that, for~every integer $1< k\leq n$, $J_k$ is a compact real interval, $J_{k-1}<J_{k}$, $\diam(J_k)\leq \varepsilon$ and $\dist(J_{k-1}, J_{k})\geq\varepsilon$.
	
	Clearly, if $n=1$ and $J(=J_1)$ is a compact real interval, then
	\begin{equation*}
	\#\mathcal{I}_1(\varepsilon)=
	\begin{cases}
	0 & \text{if }\diam(J)\leq\varepsilon,\\
	1 & \text{if }\varepsilon<\diam(J).
	\end{cases}
	\end{equation*}
\end{rem}

For $n>1$, we derive the following result giving an upper boundary for the~cardinality of $\mathcal{I}_n(\varepsilon)$ where $\varepsilon>0$.

\begin{prop}\label{Prop:border}
	Let $n\in\mathbb{N}\setminus\{1\}, \varepsilon>0$ and $J_1<J_2<\dots<J_n$ be compact real intervals. Then $$\# \mathcal{I}_n(\varepsilon) \leq 4\cdot(n-1).$$
	Moreover, the estimate is optimal.
\end{prop}
\begin{proof}
	From~the~definition of~$\mathcal{I}_n(\varepsilon)$ and~\eqref{eq:ab_ba} we get
	\begin{equation}\label{eq:aa_2ab}
	\begin{split}
	\# \mathcal{I}_n(\varepsilon)&= \#\bigl\{(a,a):\ 1\leq a\leq n,\quad \varepsilon<\diam(J_a)\bigr\}\\
	&+ 2\#\bigl\{(a,b)\colon\ 1\leq a<b\leq n,\quad \dist(J_a,J_b)<\varepsilon<\diam(J_a\cup J_b)\bigr\}.
	\end{split}
	\end{equation}
	Order all the pairs $(a,b)$ from $\mathcal{I}_n(\varepsilon)$ with $a\leq b$ lexicographically. Then, by~Lemmas~\ref{L:amongI}~and~\ref{L:bcNot}, all the pairs $(a,b)\in \mathcal{I}_n(\varepsilon)$ with $a\leq b$ can be written as
	\begin{equation}\label{eq:lexicografic}
	\begin{split}
	&(a_1,b_1)<(a_1,b_1+1)<\dots<(a_1,c_1)\\
	<\ & (a_2,b_2)<(a_2,b_2+1)<\dots<(a_2,c_2)\\
	<\ & \dots\\
	<\ & (a_k,b_k)<(a_k,b_k+1)<\dots<(a_k,c_k)
	\end{split}
	\end{equation}
	where
	\begin{enumerate}[label=(\roman*), ref=(\roman*)]
		\item\label{it:4enum_abc_1} $k = \#\left\{a\colon \left(a,b\right)\in\mathcal{I}_n(\varepsilon) \text{ for some } b\geq a\right\}$;
		\item\label{it:4enum_abc_2} $1\leq a_1<a_2<\dots<a_k\leq b_k\leq c_k\leq n$;
		\item\label{it:4enum_abc_3} $a_i\leq b_i\leq c_i\leq b_{i+1}$ for every $i\in\{1,2,\dots,k-1\}$.
	\end{enumerate}
	Now we distinguish four cases; in each of them we show that the cardinality estimate from the assertion is true.
	
	\smallskip
	\emph{Case~1:}
	If $a_1 = b_1$ and $a_k = b_k = c_k$, then the number of these pairs is
	\begin{equation*}
	1+\left(c_1-a_1\right)+(c_2-b_2+1)+\dots+(c_{k-1}-b_{k-1}+1)+1.
	\end{equation*}
	So, by~\eqref{eq:aa_2ab} and~\ref{it:4enum_abc_1}-\ref{it:4enum_abc_3},
	\begin{equation*}
	\begin{split}
	\#\mathcal{I}_n(\varepsilon)&\leq 2+2\cdot\bigl[(c_1-a_1)+(c_2-b_2+1)+\dots+(c_{k-1}-b_{k-1}+1)\bigr]\\
	&\leq 2+2\cdot\bigr[c_{k-1}-a_1+(k-2)\bigl] \leq 2+2\cdot\bigl[n-1+(n-2)\bigr] = 4\cdot(n-1).
	\end{split}
	\end{equation*}

	\smallskip
	\emph{Case~2:}
	If $a_1=b_1$ and $a_k<c_k$ (and therefore $k\leq n-1$ by~\ref{it:4enum_abc_2}), then the~number of these pairs is
	\begin{equation*}
	1+\left(c_1-a_1\right)+(c_2-b_2+1)+\dots+(c_{k}-b_{k}+1).
	\end{equation*}
	Using~\eqref{eq:aa_2ab} and~\ref{it:4enum_abc_1}-\ref{it:4enum_abc_3},
	\begin{equation*}
	\begin{split}
	\#\mathcal{I}_n(\varepsilon)&\leq 1+2\cdot\bigl[(c_1-a_1)+(c_2-b_2+1)+\dots+(c_{k}-b_{k}+1)\bigr]\\
	&\leq 1+2\cdot\bigr[c_{k}-a_1+(k-1)\bigl] \leq 1+2\cdot\bigl[n-1+(n-2)\bigr] < 4\cdot(n-1).
	\end{split}
	\end{equation*}

	\smallskip
	\emph{Case~3:}
	If $a_1<b_1$ and $a_k=b_k=c_k$ (and therefore $b_1\geq 2$ by~\ref{it:4enum_abc_2} and~\ref{it:4enum_abc_3}), then the number of these pairs is
	\begin{equation*}
	(c_1-b_1+1)+(c_2-b_2+1)+\dots+(c_{k-1}-b_{k-1}+1)+1.
	\end{equation*}
	Therefore, \eqref{eq:aa_2ab} and~\ref{it:4enum_abc_1}-\ref{it:4enum_abc_3} imply
	\begin{equation*}
	\begin{split}
	\#\mathcal{I}_n(\varepsilon)&\leq 1+2\cdot\bigl[(c_1-b_1+1)+(c_2-b_2+1)+\dots+(c_{k-1}-b_{k-1}+1)\bigr]\\
	&\leq 1+2\cdot\bigr[c_{k-1}-b_1+(k-1)\bigl] \leq 1+2\cdot\bigl[n-2+(n-1)\bigr] < 4\cdot(n-1).
	\end{split}
	\end{equation*}

	\smallskip
	\emph{Case~4:}
	If $a_1<b_1$ and $a_k<c_k$ (and therefore $k\leq n-1$ by~\ref{it:4enum_abc_2}, and~$b_1\geq 2$ by~\ref{it:4enum_abc_2} and~\ref{it:4enum_abc_3}), then the number of these pairs is
	\begin{equation*}
	(c_1-b_1+1)+(c_2-b_2+1)+\dots+(c_{k}-b_{k}+1).
	\end{equation*}
	Hence, \eqref{eq:aa_2ab} and~\ref{it:4enum_abc_1}-\ref{it:4enum_abc_3} give
	\begin{equation*}
	\begin{split}
	\#\mathcal{I}_n(\varepsilon)&\leq 2\cdot\bigl[(c_1-b_1+1)+(c_2-b_2+1)+\dots+(c_{k}-b_{k}+1)\bigr]\\
	&\leq 2\cdot\bigr(c_{k}-b_1+k\bigl) \leq 2\cdot\bigl[n-2+(n-1)\bigr] < 4\cdot(n-1).
	\end{split}
	\end{equation*}

	\smallskip
	We proved that in every case one has $\#\mathcal{I}_n(\varepsilon)\leq 4\cdot(n-1)$. To finish the proof we need to show that for all $n\in\mathbb{N}\setminus\{1\}$ and $\varepsilon>0$ there are compact real intervals $J_1<J_2<\dots<J_n$ such that $\#\mathcal{I}_n(\varepsilon) = 4\cdot(n-1)$. The construction is as follows. Let $J_1<J_n$ be compact real intervals with $\dist\left(J_1, J_n\right)<\varepsilon$ and $\diam(J_1)>\varepsilon$, $\diam(J_n)>\varepsilon$. Take any compact real intervals $J_2, \dots, J_{n-1}$ such that $J_1<J_2<\dots<J_{n-1}<J_n$. Immediately, $(1,1), (1,n), (n,1), (n,n) \in\mathcal{I}_n(\varepsilon)$. Therefore, by~Lemma~\ref{L:amongI} and~\eqref{eq:ab_ba}, also $(1,k), (k,1), (k,n), (n,k) \in\mathcal{I}_n(\varepsilon)$ for every $k\in\{2,3,\dots,n-1\}$. Moreover, by~Lemma~\ref{L:bcNot}, there are no other pairs in $\mathcal{I}_n(\varepsilon)$. Hence, $\#\mathcal{I}_n(\varepsilon) = 4+4\cdot(n-2) = 4\cdot(n-1)$.
\end{proof}

\section{Correlation sums - solenoidal case}\label{S:solenoidal}
Let $(I, f)$ be a dynamical system. For $x\in I$, we say that $\omega_f(x)$ is \emph{solenoidal} (see \cite[p.~4]{Blo95}) if there are a sequence of integers $2\leq p_0<p_1<\dots$ and a sequence of~nondegenerate closed intervals $J_0\supseteq J_1\supseteq\dots\in I$ such that every $J_t$ is $p_t$-periodic and $\omega_f(x)\subseteq Q$ where
$$Q=\bigcap_{t=0}^\infty Q_t,\qquad Q_t = \bigsqcup_{i=0}^{p_t-1} f^{i}\left(J_t\right).$$

Put $q_0=p_0$ and $q_t=p_t/p_{t-1}$ for $t\geq1$. Denote the Cartesian products of $\left\{0,1,\dots, q_{i}-1\right\}$ by $\Sigma$ and $\mathcal{A}^t$:
$$\Sigma = \prod_{i=0}^\infty \left\{0,1,\dots, q_{i}-1\right\}, \qquad\mathcal{A}^t = \prod_{i=0}^{t-1}\left\{0,1,\dots, q_{i}-1\right\}\quad (t\geq1).$$
Every element $a=a_0a_1\dots a_{t-1}$
of $\mathcal{A}^t$ is called a \emph{word} and its \emph{length} is $\vert a\vert=t$.
Define $\mathcal{A}^0=\{o\}$ (a singleton set containing the empty word $o$)
and $\mathcal{A}^*=\bigsqcup_{t\ge 0} \mathcal{A}^t$.
Let $\pi_t:\Sigma\to\mathcal{A}^t$ ($t\ge 0$)
be the natural projection onto the first $t$ coordinates.
For $a\in\mathcal{A}^t$, let $[a]$ denotes the set of all sequences $\alpha\in\Sigma$
and all words $b\in\mathcal{A}^*$
\emph{starting with $a$} (i.e.,~$\pi_t(\alpha)=\pi_t(b)=a$).

On $\Sigma$ and on every $\mathcal{A}^t$ define
addition in a natural way with carry from left to~right; the sets
$\Sigma$ and $\mathcal{A}^t$ equipped with this operation are abelian groups.
Identify $10^\infty\in\Sigma$ and every $10^{t-1}\in\mathcal{A}^t$ ($t\ge 1$) with integer $1$,
and inductively define
$\alpha+n$, $a+n$ for $\alpha\in\Sigma$, $a\in\mathcal{A}^t$, and $n\in\mathbb{Z}$.

For $t\ge 0$ write
$$
Q_t = \bigsqcup_{a\in\mathcal{A}^t} K_a
\qquad
\text{where } K_{0^t+i} =  f^i(J_t)
\text{ for every } i\in[0,p_t).
$$
Notice that every $K_a$ ($a\in\mathcal{A}^t)$ is a nondegenerate closed
$p_t$-periodic interval $[y_a,z_a]$,
and $K_{b}\subseteq K_{a}$ for every $b\in[a]$.
We can also write
$$
Q = \bigsqcup_{\alpha\in\Sigma} K_\alpha
\qquad\text{where }
K_\alpha = \bigcap_{t=0}^\infty K_{\pi_t(\alpha)}
\text{ for every } \alpha\in\Sigma.
$$
Here, every $K_\alpha$ is either a singleton $\{y_\alpha\}$ or a
nondegenerate closed interval $[y_\alpha,z_\alpha]$.

Fix a continuous map $f:I\to I$ and a point $x\in I$, and put $x_i=f^i(x)$ for every $i\geq 0$.
Take $m\in\mathbb{N}$ and $t\geq 0$. For all $a,b\in\mathcal{A}^t$, put
\begin{equation*}
\begin{split}
\dist_m(K_a,K_b) &= \max_{i\in\mathbb{Z}_m} \dist\left(K_{a+i},K_{b+i}\right),\\
\diam_m(K_a,K_b) &= \max_{i\in\mathbb{Z}_m} \diam\left(K_{a+i}\cup K_{b+i}\right).
\end{split}
\end{equation*}
Note that, due to $p_t$-periodicity of intervals $K_a$, $\dist_m=\dist_{p_t}$ and $\diam_m=\diam_{p_t}$
for every $m\geq p_t$.
For $\varepsilon>0$, define
\begin{equation*}
\begin{split}
	N_m(x,t,\varepsilon) &= \left\{(a,b)\in \mathcal{A}^t\times\mathcal{A}^t\colon \dist_m(K_a,K_b) < \varepsilon\right\},
	\\
	N_m^\circ(x,t,\varepsilon) &= \left\{(a,b)\in \mathcal{A}^t\times\mathcal{A}^t\colon \diam_m(K_a, K_b) \le \varepsilon\right\}.
\end{split}
\end{equation*}

The following lemma is a combination of~Lemmas~2.3~and~4.2 from~\cite{Spi19}.

\begin{lemma}\label{L:solenoidal}
	Let $x\in I$ be such that $\omega_f(x)$ is solenoidal and $Q=\bigcap Q_t$ be as defined above. Then,
	\begin{enumerate}[label=(\alph*), ref=(\alph*)]
		\item\label{it:L:solenodial_iterate} for every $t$, there is $n_0$ such that $f^n(x)\in Q_t$ for every $n\ge n_0$;
		\item\label{it:L:solenodial_boundary_c} for all $m\in\mathbb{N}, t\geq0$ and $\varepsilon>0$,
		$$N_m^\circ(x,t,\varepsilon)\cdot p_t^{-2} \leq \underline{c}_m(x,\varepsilon)\leq \overline{c}_m(x,\varepsilon) \leq N_m(x,t,\varepsilon)\cdot p_t^{-2}.$$
	\end{enumerate}
\end{lemma}

Our main result of the present paper concerning correlation sums says that if $\omega$-limit set of $x$ is solenoidal, then the asymptotic correlation sums \emph{always} exist. The~theorem also presents a formula for their computation with respect to distances or to diameters of intervals forming the $\omega$-limit set. Since the proof uses the result from~Proposition~\ref{Prop:border}, one need to assume that a metric $\varrho$ is order-preserving and compatible with the Euclidean topology. \begin{thm}\label{Thm:solenoidal}
	Let $I$ be equipped with an order-preserving metric compatible with the Euclidean topology. Let $(I,f)$ be a dynamical system and assume that $x\in I$ is such that $\omega_f(x)$ is solenoidal. Then, for all $m\in\mathbb{N}$ and $\varepsilon>0$,
	$$\underline{c}_m(x,\varepsilon) = \overline{c}_m(x,\varepsilon);$$
	i.e., the asymptotic correlation sum exists. Moreover,
	$$c_m(x,\varepsilon) = \lim\limits_{t\to\infty}\frac{\#N_m^\circ(x,t,\varepsilon)}{p_t^{2}} = \lim\limits_{t\to\infty}\frac{\#N_m(x,t,\varepsilon)}{p_t^{2}}.$$
\end{thm}
\begin{proof}
	Fix $m, t\in\mathbb{N}$, $\varepsilon>0$ and $x\in I$ with solenoidal $\omega_f(x)$. We use the notation as above. By Lemma~\ref{L:solenoidal}\ref{it:L:solenodial_iterate}, there is a nonnegative integer $n_0$ such that $f^n(x)\in Q_t$ for every $n\ge n_0$. Therefore, by Lemma~\ref{L:solenoidal}\ref{it:L:solenodial_boundary_c}, the inequalities
	$$\#N_m^\circ\left(x,t,\varepsilon\right)\cdot p_t^{-2} \leq \underline{c}_m\left(f^n(x),\varepsilon\right)\leq \overline{c}_m\left(f^n(x),\varepsilon\right) \leq \#N_m\left(x,t,\varepsilon\right)\cdot p_t^{-2}$$
	hold for every $n\ge n_0$. Moreover, by Lemma~\ref{L:iteration}, asymptotic correlation sums do not depend on the first finitely many iterates, so we have
	\begin{equation}\label{eq:Thm_solenoidal_c_boundary}
		\#N_m^\circ\left(x,t,\varepsilon\right)\cdot p_t^{-2} \leq \underline{c}_m\left(x,\varepsilon\right)\leq \overline{c}_m\left(x,\varepsilon\right) \leq \#N_m\left(x,t,\varepsilon\right)\cdot p_t^{-2}.
	\end{equation}
	
	For $i\in\mathbb{Z}_m$, put
	\begin{equation*}
	A_i = \left\{(a,b)\in\mathcal{A}^t\times\mathcal{A}^t\colon \dist_m\left(K_a, K_b\right) <\varepsilon < \diam\left(K_{a+i}\cup K_{b+i}\right) \right\}.
	\end{equation*}
	Trivially,
	\begin{equation}\label{eq:Thm_solenoidal_c_cupA}
		N_m(x,t,\varepsilon) \setminus N_m^\circ(x,t,\varepsilon) = \bigcup_{i=0}^{m-1}A_i.
	\end{equation}
	Moreover,
	\begin{equation*}
	A_i \subseteq \left\{(a,b)\in\mathcal{A}^t\times\mathcal{A}^t\colon \dist\left(K_{a+i}, K_{b+i}\right) <\varepsilon < \diam\left(K_{a+i}\cup K_{b+i}\right) \right\}.
	\end{equation*}
	for every $i\in\mathbb{Z}_m$, and therefore, by~Proposition~\ref{Prop:border},
	\begin{equation*}
	\#\left(\bigcup_{i=0}^{m-1}A_i\right) \leq 4m\left(p_t-1\right).
	\end{equation*}
	So, by~\eqref{eq:Thm_solenoidal_c_boundary} and~\eqref{eq:Thm_solenoidal_c_cupA},
	\begin{equation*}
	\begin{split}
	0 &\leq \overline{c}_m\left(x,\varepsilon\right) -\underline{c}_m\left(x,\varepsilon\right) \leq \limsup\limits_{t\to\infty}\left\vert\frac{\#N_m(x,t,\varepsilon) -\#N_m^\circ(x,t,\varepsilon)}{p_t^{2}}\right\vert\\
	&\leq \limsup\limits_{t\to\infty}\left\vert\frac{4m\left(p_t-1\right)}{p_t^{2}}\right\vert = 0.
	\end{split}
	\end{equation*}
\end{proof}

\section{Correlation sums - finite case}\label{S:finite}
In this section, we study asymptotic correlation sums for points having finite $\omega$-limit sets. For a similar result, see~\cite[Lemma~4.1]{Maj16} and~\cite[Proposition~4.1]{Spi19}. 

\begin{thm}\label{Thm:finite_c}
	Let $(I,f)$ be a dynamical system and $x\in I$ be such that $\omega_f(x)$ is finite. Then, for every $m\in\mathbb{N}$ and all but finitely many $\varepsilon>0$,
	$$\underline{c}_m(x,\varepsilon) = \overline{c}_m(x,\varepsilon) .$$
	More precisely, if $\omega_f(x) = \left\{y_0, y_1,\dots, y_{p-1}\right\}$ and
	$$0< \varepsilon\not\in\left\{\varrho_m\left(y_i, y_j\right)\colon 0\leq i,j<p\right\},$$ then the asymptotic correlation sum exists and
	\begin{equation}\label{eq:Thm:finite_c_value}
		c_m(x,\varepsilon) = \frac1{p^2}\#\left\{(i,j)\in\mathbb{Z}_p\times\mathbb{Z}_p\colon\ \varrho_m\left(y_i, y_j\right)\leq\varepsilon\right\}.
	\end{equation}
\end{thm}

\begin{proof}
	Let $m\in\mathbb{N}$ and $x\in I$ be such that $\omega_f(x)$ has finite cardinality $p\in\mathbb{N}$. Since the elements of $\omega_f(x)$ forms a periodic orbit, we can denote them by $y_0, y_1,\dots, y_{p-1}$ such that
	\begin{equation}\label{eq:Thm:finite_c_orbit}
		\lim\limits_{n\to\infty}f^{pn+i}(x) = y_i \quad\text{ for every $i\in \mathbb{Z}_p$}.
	\end{equation}
	
	Fix arbitrary $\varepsilon>0$ and $i,j\in \mathbb{Z}_p$, and put $\delta = \varrho_m\left(y_i, y_j\right)$. Suppose that $\delta<\varepsilon$. By~\eqref{eq:Thm:finite_c_orbit}, there is $n_0$ such that
	\begin{equation*}
		\varrho_m\left(f^{pn+i}(x),y_i\right)<\frac{\varepsilon-\delta}{2} \quad\text{ and } \quad \varrho_m\left(f^{pn+j}(x),y_j\right)<\frac{\varepsilon-\delta}{2}
	\end{equation*}
	for every $n\geq n_0$. Thus, for any $k,l\geq n_0$, we have
	\begin{equation*}
		\varrho_m\left(f^{pk+i}(x), f^{pl+j}(x)\right) \leq \varrho_m\left(f^{pk+i}(x), y_i\right) + \varrho_m\left(y_i, y_j\right) + \varrho_m\left(f^{pl+j}(x), y_j\right) < \varepsilon.
	\end{equation*}
	
	On the other hand, if $\delta>\varepsilon$ then, by~\eqref{eq:Thm:finite_c_orbit}, there is $n_0^\prime$ such that
	\begin{equation*}
	\varrho_m\left(f^{pn+i}(x),y_i\right) < \frac{\delta-\varepsilon}{2} \quad\text{ and } \quad \varrho_m\left(f^{pn+j}(x),y_j\right) < \frac{\delta-\varepsilon}{2}
	\end{equation*}
	for every $n\geq n_0^\prime$. Thus, for any $k,l\geq n_0^\prime$, we have
	\begin{equation*}
	\varrho_m\left(f^{pk+i}(x), f^{pl+j}(x)\right) \geq \varrho_m\left(y_i, y_j\right) - \varrho_m\left(f^{pk+i}(x), y_i\right) - \varrho_m\left(f^{pl+j}(x), y_j\right) > \varepsilon.
	\end{equation*}
	So, for any $i,j\in \mathbb{Z}_p$, we get $\varrho_m\left(y_i, y_j\right) <\varepsilon$ if and only if $\varrho_m\left(f^{pk+i}(x), f^{pl+j}(x)\right) <\varepsilon$ for all sufficiently large $k,l\geq0$. This result and Lemma~\ref{L:iteration} immediately imply the~assertion.
\end{proof}

\begin{rem}\label{rem:singleton-EPer}
	By~Theorem~\ref{Thm:finite_c}, if $x\in I$ is such that $\omega_f(x)$ is a singleton, then the~asymptotic correlation sum $c_m(x,\varepsilon)$ exists and is equal to $1$ for \emph{every} $\varepsilon>0$. Further, it is trivial to show that also for every $x\in\EPer(f)$ the asymptotic correlation sum $c_m(x,\varepsilon)$ exists and~\eqref{eq:Thm:finite_c_value} is valid for all $m\in\mathbb{N}$ and $\varepsilon>0$. However, in general, Theorem~\ref{Thm:finite_c} cannot be strengthened to hold for \emph{every} $\varepsilon>0$, as is shown in~the~following proposition.
\end{rem}

\begin{prop}\label{Prop:c}
	There are a continuous map $f\colon I\to I$ with zero topological entropy (in fact, of type $2$ for Sharkovsky's order), a point $x_0\in I$ with finite $\omega_f\left(x_0\right)$, $m\in\mathbb{N}$, and $\varepsilon>0$ such that $$\underline{c}_m\left(x_0,\varepsilon\right) < \overline{c}_m\left(x_0,\varepsilon\right).$$
\end{prop}

\begin{proof}
Put $y_0=1/4$, $y_1=3/4$ and $\varepsilon = 1/2$. Fix a point $x_0\in\left(0, y_0\right)$. We construct a continuous map $f\colon I\to I$ with zero topological entropy such that $\omega_f\left(x_0\right) = \left\{y_0,y_1\right\}$, and then we show that $\underline{c}_1\left(x_0,\varepsilon\right) < \overline{c}_1\left(x_0,\varepsilon\right)$ (note that this does not contradict~Theorem~\ref{Thm:finite_c} since $\varepsilon=\left\vert y_1-y_0\right\vert$). We proceed in several steps.

\smallskip
\emph{Step~1. Construction of intervals $I_n, J_n$.} In this step, we construct intervals whose union will contain the trajectory of $x_0$.
\smallskip 

Inductively construct compact intervals $I_n=\left[a_n, b_n\right], J_n=\left[c_n, d_n\right]$ ($n\geq0$) such that $x_0\in I_0$ and, for every $n\in\mathbb{N}_0$,
\begin{enumerate}[label=(\roman*), ref=(\roman*)]
	\item\label{Pr:it:subset} $I_n\subset \left(0,y_0\right)$, and $J_n\subset \left(y_0,y_1\right)$;
	\item\label{Pr:it:dist} $\dist\left(I_{2n}, J_{2n}\right)>\varepsilon$;
	\item\label{Pr:it:diam} $\diam\left(I_{2n+1}\cup J_{2n+1}\right)<\varepsilon$;
	\item\label{Pr:it:order} $I_0 < I_1 < \dots < I_n$ and $J_0 < J_1 < \dots < J_n$;
	\item\label{Pr:it:lim} $\lim\limits_{n\to\infty}a_n=\lim\limits_{n\to\infty}b_n = y_0$, and $\lim\limits_{n\to\infty}c_n=\lim\limits_{n\to\infty}d_n = y_1$.
\end{enumerate}
Note that, for $s<t$, $\dist\left(I_s, J_t\right)>\dist\left(I_{s+1}, J_{s+1}\right)$ and $\dist\left(I_s, J_t\right)>\dist\left(I_{s}, J_{s}\right)$. Therefore, by~\ref{Pr:it:dist},
\begin{equation}\label{eq:Prop:dist}
\dist\left(I_s,J_t\right)>\varepsilon \qquad\text{for any } s<t.
\end{equation}
Similarly, for any $s>t$, $\diam\left(I_s\cup J_t\right) < \diam\left(I_s\cup J_s\right)$ and $\diam\left(I_s\cup J_t\right) < \diam\left(I_{s-1}\cup J_{s-1}\right)$. Therefore, by~\ref{Pr:it:diam},
\begin{equation}\label{eq:Prop:diam}
\diam\left(I_s\cup J_t\right)<\varepsilon \qquad\text{for any } s>t. 
\end{equation}

\smallskip
\emph{Step~2. Construction of points~$x_n$ in $I$.} Here, we construct points that will form the trajectory of $x_0$.
\smallskip

For every $k\in\mathbb{N}$, fix $2^k$ points in $I_k$. From left to right, denote these points by~$x_{2i}$ where $i\in\left\{2^k-1,\dots, 2\cdot\left(2^k-1\right)\right\}$. Analogously, for every $k\in\mathbb{N}_0$, fix $2^k$ points in $J_k$, and, from left to right, denote these points by~$x_{2i+1}$ where $i\in\left\{2^k-1,\dots, 2\cdot\left(2^k-1\right)\right\}$. Trivially, by~\ref{Pr:it:order}~and~\ref{Pr:it:lim}, $\left(x_{2i}\right)_i$ increases to $y_0$ and $\left(x_{2i+1}\right)_i$ increases to $y_1$.

\smallskip
\emph{Step~3. Definition of a continuous map $f\colon I\to I$.} In this step, we define a~continuous map $f\colon I\to I$. We use the notation $f\colon[a,b]\nearrow[c,d]$ to denote that $f(a)=c, f(b)=d$, and $f$ is continuous and increasing on $[a,b]$. Analogously, we write $f\colon[a,b]\searrow[c,d]$ to denote that $f(a)=d, f(b)=c$, and $f$ is continuous and decreasing on $[a,b]$.
\smallskip

Define $f\colon I\to I$ such that:
\begin{itemize}
	\item $f\colon [x_0, y_0]\nearrow [x_1,y_1]$ where $f\left(x_n\right) = x_{n+1}$ for every even $n\in\mathbb{N}_0$;
	\item $f\colon [x_1, y_1]\nearrow [x_2,y_0]$ where $f\left(x_n\right) = x_{n+1}$ for every odd $n\in\mathbb{N}_0$;
	\item $f(x)=x_1$ for every $x\in\left[0,x_0\right)$;
	\item $f(x)=y_0$ for every $x\in\left(y_1,1\right]$;
	\item $f\colon [y_0, x_1]\searrow [x_2,y_1]$ is linear with constant slope.
\end{itemize}
Obviously, such a continuous map $f$ exists. Moreover, since $f\left(\left[y_0, x_1\right]\right) = \left[x_2, y_1\right] \supset \left[y_0, x_1\right]$, $f$ has a fixed point $z\in \left(y_0, x_1\right)$.

\smallskip
\emph{Step~4. Proof of the fact that $\omega_f\left(x_0\right)$ is finite and entropy of $f$ is zero.} We show that
\begin{equation}\label{eq:Prop:omega}
	\omega_f(x)=\left\{y_0,y_1\right\}\qquad\text{for any } x\in I\setminus\{z\}.
\end{equation}
Therefore, $f$ is of type $2$ for Sharkovsky's order and hence not Li-Yorke chaotic with zero entropy.
\smallskip

Fix $x\in I$. We distinguish the following cases for $x$; in each of them we show that \eqref{eq:Prop:omega} holds.
\begin{itemize}
	\item If $x\in \left\{y_0, y_1\right\}$, then \eqref{eq:Prop:omega} holds trivially.
	\item If $x>y_1$, then $f(x)=y_0$, and so $\omega_f(x)=\omega_f\left(y_0\right)=\left\{y_0, y_1\right\}$.
	\item If $x\in \left[x_0, y_0\right)$, then $f^n(x)\in \left[x_n, y_{n \mmod 2}\right)$ for every $n\in\mathbb{N}_0$. Similarly, if $x\in \left[x_1, y_1\right)$, then $f^n(x)\in \left[x_{n+1}, y_{n+1 \mmod 2}\right)$ for every $n\in\mathbb{N}_0$. Therefore, in~both cases, $\omega_f(x) = \left\{y_0,y_1\right\}$.
	\item If $x<x_0$, then $f(x)=x_1$, and so $\omega_f(x)=\omega_f\left(x_1\right)=\left\{y_0, y_1\right\}$.
	\item Finally, suppose that $x\in \left(y_0, x_1\right)\setminus\{z\}$. On $\left(y_0, x_1\right)$, $f$ is linear with constant slope $\gamma$. Therefore, for every $n\in\mathbb{N}$ with $f^n(x)\in \left(y_0, x_1\right)$, the equality $\left\vert f^n(x)-z\right\vert = \vert\gamma\vert^n\cdot\left\vert x-z\right\vert$ holds. Since $f\left(y_0\right)=y_1$ and $f\left(x_1\right)=x_2$, $\gamma=\left(y_1-x_2\right) / \left(y_0-x_1\right)$. Moreover, $y_1-x_2>x_1-y_0$, and so $\gamma<-1$. This means that there is $n\in\mathbb{N}$ such that $f^n(x)\not\in \left(y_0, x_1\right)$. Therefore, with~respect to the previous cases, \eqref{eq:Prop:omega} holds.
\end{itemize}

\smallskip
\emph{Step~5. Proof of the fact that $\underline{c}_1\left(x_0,\varepsilon\right) < \overline{c}_1\left(x_0,\varepsilon\right)$; recall that $\varepsilon=\left\vert y_1-y_0\right\vert=1/2$.} In this step, we prove that $\underline{c}_1\left(x_0,\varepsilon\right)\leq 7/10$ and $8/10\leq\overline{c}_1\left(x_0,\varepsilon\right)$. Therefore, $\underline{c}_1\left(x_0,\varepsilon\right) < \overline{c}_1\left(x_0,\varepsilon\right)$.
\smallskip

By~\ref{Pr:it:dist},~\ref{Pr:it:diam},~\eqref{eq:Prop:dist} and~\eqref{eq:Prop:diam}, we have that for $i,j\in\mathbb{N}_0$, $\left\vert x_i-x_j\right\vert\leq \varepsilon$ if and only if one of the following conditions holds:
\begin{itemize}
	\item $i,j$ are both even or both odd;
	\item there is an odd $s$ such that either $x_i\in I_s$ and $x_j\in J_s$, or vice versa;
	\item there are $s,t$ such that $s>t$ and either $x_i\in I_s$ and $x_j\in J_t$, or vice versa.
\end{itemize}
Using these conditions, we derive a formula for $C_1(x_0,n,\varepsilon)$ with $n=2\cdot\left(2^{k+1}-1\right)$, $k\in\mathbb{N}$:
\begin{equation*}
\begin{split}
	C_1(x_0,n,\varepsilon) &= \frac{1}{n^2}\left[ 2\cdot\left(\frac{n}{2}\right)^2 + 2\sum_{\substack{ {s\leq k}\\ {s \text{ is odd}}}} 2^{2s} + 2\sum_{s=1}^k \sum_{t=0}^{s-1}2^{s+t} \right]\\
	&= \frac{2}{n^2}\left[\left(2^{k+1}-1\right)^2 + \frac{4}{15}\left(16^{\left\lceil{\frac{k}{2}}\right\rceil}-1 \right) + \frac{4^{k+1}-3\cdot2^{k+1}+2}{3} \right].
\end{split}
\end{equation*}
Hence, for $k=2l$ and $n=2\cdot\left(2^{k+1}-1\right)$, we have
\begin{equation*}
\underline{c}_1\left(x_0,\varepsilon\right) \leq \lim\limits_{l\to\infty} \frac{2}{n^2}\left[\left(2^{2l+1}-1\right)^2 + \frac{4}{15}\left(16^{l}-1 \right) + \frac{4^{2l+1}-3\cdot2^{2l+1}+2}{3} \right]=\frac{7}{10},
\end{equation*}
and for $k=2l+1$ and $n=2\cdot\left(2^{k+1}-1\right)$,
\begin{equation*}
\begin{split}
\overline{c}_1\left(x_0,\varepsilon\right) &\geq \lim\limits_{l\to\infty} \frac{2}{n^2}\left[\left(2^{2l+2}-1\right)^2 + \frac{4}{15}\left(16^{l+1}-1 \right) + \frac{4^{2l+2}-3\cdot2^{2l+2}+2}{3} \right]\\
&= \frac{8}{10}.
\end{split}
\end{equation*}
Therefore, $\underline{c}_1\left(x_0,\varepsilon\right) < \overline{c}_1\left(x_0,\varepsilon\right)$.
\end{proof}

\begin{rem}
	In~Step~4 of the previous proof, we showed that $f$ is of type $2$ for Sharkovsky's order. Note that, by~Theorem~\ref{Thm:finite_c}, no map $f$ of type $1$ for Sharkovsky's order satisfies the assertion of Proposition~\ref{Prop:c}.
\end{rem}

\section{Recurrence determinism}\label{S:determinism}

The following theorem says that trajectories of points with finite $\omega$-limit sets under a continuous interval map are perfectly predictable in the -- arbitrarily large -- finite horizon (see also~\cite[Lemma~4.2]{Maj16} and~\cite[Proposition~4.1]{Spi19}).

\begin{thm}
	Let $x\in I$ be such that $\omega_f(x)$ is finite. Then, for every $m\in\mathbb{N}$,
	$$\lim\limits_{\varepsilon\to 0}\rdet_m(x,\varepsilon) = 1.$$ More precisely, if $\omega_f(x) = \left\{y_0, y_1,\dots, y_{p-1}\right\}$ and
	\begin{equation}\label{eq:T:det_eps}
	0<\varepsilon < \min\left\{\varrho\left(y_i,y_j\right)\colon i\neq j\right\},
	\end{equation}
	then the asymptotic recurrence $m$-determinism exists and
	$$\rdet_m(x,\varepsilon) = 1.$$
\end{thm}

\begin{proof}
	Let $m\in\mathbb{N}$ and $x\in I$ be such that $\omega_f(x) = \left\{y_0,y_1,\dots,y_{p-1}\right\}$ for some $p\in\mathbb{N}$, and fix $\varepsilon>0$ such that \eqref{eq:T:det_eps} holds. Trivially,
	$$\left\{\varrho_m\left(y_i,y_j\right)\colon i\neq j\right\} \subseteq \left\{\varrho\left(y_i,y_j\right)\colon i\neq j\right\},$$
	and so
	\begin{equation}\label{eq:T:det_min}
		\min\left\{\varrho_m\left(y_i,y_j\right)\colon i\neq j\right\} \geq \min\left\{\varrho\left(y_i,y_j\right)\colon i\neq j\right\}>\varepsilon>0.
	\end{equation}
	Using this fact, the asymptotic correlation sums $c_1(x,\varepsilon)$ and $c_m(x,\varepsilon)$ exist by~Theorem~\ref{Thm:finite_c}, and
	\begin{equation*}
		c_1(x,\varepsilon) = \#A_1\cdot p^{-2}, \qquad c_m(x,\varepsilon) = \#A_m\cdot p^{-2}
	\end{equation*}
	where
	\begin{equation}\label{eq:T:det_A}
	\begin{split}
	A_1 &= \left\{(i,j)\in\mathbb{Z}_p\times\mathbb{Z}_p\colon\ \varrho\left(y_i, y_j\right)\leq\varepsilon\right\},\\
	A_m &= \left\{(i,j)\in\mathbb{Z}_p\times\mathbb{Z}_p\colon\ \varrho_m\left(y_i, y_j\right)\leq\varepsilon\right\}.
	\end{split}
	\end{equation}
	Hence,
	\begin{equation}\label{eq:T:det_frac}
	\rdet_m(x,\varepsilon) = \frac{c_m(x,\varepsilon)}{c_1(x,\varepsilon)} = \frac{\#A_m}{\#A_1}.
	\end{equation}
	Moreover, from~\eqref{eq:T:det_min} and~\eqref{eq:T:det_A} we get $\#A_1 = \#A_m = p$.
	Now~\eqref{eq:T:det_frac} yields the~assertion.
\end{proof}

Now, let us introduce some notation required in the next proposition, where we give an example of a not Li-Yorke chaotic interval map $f$ and a point $x$ with the solenoidal $\omega$-limit set such that
 \begin{equation*}
 	\liminf\limits_{\varepsilon\to0}\rdet_m(x,\varepsilon)\leq2/3
 	\qquad\text{and}\qquad
 	\liminf\limits_{\varepsilon\to0}\lim\limits_{n\to\infty}\DET_m(x,n,\varepsilon) \leq2/3
 \end{equation*}
for any positive integer $m$ greater than one.

Put $\mathcal{A}^t=\{0,1\}^t$ for $t\ge 1$ and $\mathcal{A}^*=\bigcup_{t}\mathcal{A}^t$. We say that a system $\mathcal{K}=\{K_a\colon a\in\mathcal{A}^*\}$ of non-degenerate closed subintervals $K_a$ of $I$ is \emph{admissible} if the following conditions hold:
\begin{itemize}
	\item
	$\min K_0=0,\ \max K_1=1\ $ and $\ K_0<K_1$;
	\item
	$\min K_{a0}=\min K_a,\ \max K_{a1}=\max K_a\ $ and $\ K_{a0}<K_{a1}\ $ for every $a\in\mathcal{A}^*$;
	\item
	$\nu_t = \max\{\diam K_a\colon a\in\mathcal{A}^t\}$ converges to zero as $t\to\infty$.
\end{itemize}
Put $Q_t=\bigsqcup_{a\in\mathcal{A}^t} K_a$ and $Q=Q(\mathcal{K})=\bigcap Q_t$; note that $Q$ is
a Cantor set.

\begin{prop}\label{Prop:rdet}
	There are a continuous map $f\colon I\to I$ with zero topological entropy (actually, not Li-Yorke chaotic) and a point $x\in I$ with solenoidal $\omega_f\left(x\right)$ such that for every $m\in\mathbb{N}\setminus\{1\}$,
	\begin{equation*}
		\liminf\limits_{\varepsilon\to0}\rdet_m(x,\varepsilon)<1
		\qquad\text{and}\qquad
		\liminf\limits_{\varepsilon\to0}\lim\limits_{n\to\infty}\DET_m(x,n,\varepsilon) <1.
	\end{equation*}
\end{prop}

\begin{proof}
	We give an example of a continuous map $f\colon I\to I$ with zero topological entropy and the unique infinite $\omega$-limit set. Then we prove that for any $x\in I$ with such $\omega$-limit set, $\liminf\nolimits_{\varepsilon\to0}\rdet_m(x,\varepsilon)\leq 2/3$ and $\liminf\nolimits_{\varepsilon\to0}\lim\nolimits_{n\to\infty}\DET_m(x,n,\varepsilon) \leq 2/3$ for every $m\in\mathbb{N}\setminus\{1\}$. The proof is divided into 5 steps.
	
	\smallskip
	\emph{Step~1. Definition of a continuous map $f\colon I\to I$ with zero entropy.}
	\smallskip
	
	Fix an integer $r\geq5$. Let $\mathcal{K}$ be the admissible system such that for every $t\in\mathbb{N}$ and $a=a_0a_1\dots a_{t-1}\in\mathcal{A}^t$,
	\begin{equation}\label{eq:Prop:rdet:diam}
	\diam(K_a)=
	\begin{cases}
	r^{-t} & \text{if } a_0=0,\\
	2r^{-t} & \text{if } a_0=1.
	\end{cases}
	\end{equation}
	Clearly, $\mathcal{K}$ is uniquely defined. It is easy to see that
	\begin{equation}\label{eq:Prop:rdet:dist01}
	\dist\left(K_0,K_1\right)=1-3/r
	\end{equation}
	and, for any $t\in\mathbb{N}$ and $a=a_0a_1\dots a_{t-1}\in\mathcal{A}^t$,
	\begin{equation}\label{eq:Prop:rdet:dist}
	\dist\left(K_{a0}, K_{a1}\right)=
	\begin{cases}
	(r-2) r^{-(t+1)} & \text{if } a_0=0,\\
	2(r-2) r^{-(t+1)} & \text{if } a_0=1.\\
	\end{cases}
	\end{equation}
	Obviously, there is a continuous map $f\colon I\to I$ such that
	\begin{equation}\label{eq:Prop:rdet:fKa}
	f\left(K_a\right)=K_{a+1}\  \text{ for every } a\in\mathcal{A}^*.
	\end{equation}
	With respect to~\eqref{eq:Prop:rdet:diam},
	$$\lim\limits_{t\to\infty}\max\limits_{a\in\mathcal{A}^t}\diam\left(K_a\right) =0.$$
	Therefore, $f$ is conjugate to the map from~\cite{Del80}, and so $f$ is not Li-Yorke chaotic and $Q(\mathcal{K})$ is the only infinite $\omega$-limit set of $f$. Moreover, $Q(\mathcal{K})$ is solenoidal and minimal, so $\omega_f(x)=Q(\mathcal{K})$ for every $x\in Q(\mathcal{K})$.
	Fix $x\in Q(\mathcal{K})$, $k\in\mathbb{N}$, $m\in\mathbb{N}\setminus\{1\}$ and put $t=k+1$. Let $\left(\varepsilon_n\right)_{n=1}^{\infty}$ be the sequence of~reals with $\varepsilon_n=r^{-n}$ for every $n\in\mathbb{N}$. 
	
	\smallskip
	\emph{Step~2. Proof of the fact that $N_1^\circ\left(x,t,\varepsilon_k\right)=3\cdot 2^k$ for $t=k+1$.}
	\smallskip
	
	Recall that $N_1^\circ\left(x,t,\varepsilon_k\right) = \left\{(a,b)\in \mathcal{A}^t\times\mathcal{A}^t\colon\ \diam(K_a\cup K_b) \le \varepsilon_k\right\}$. We may write
	\begin{equation}\label{eq:Prop:rdet:N1}
	N_1^\circ\left(x,t,\varepsilon_k\right) = A_0\left(x,t,\varepsilon_k\right) \sqcup A_1\left(x,t,\varepsilon_k\right)
	\end{equation}
	where
	\begin{equation*}
	\begin{split}
	A_0\left(x,t,\varepsilon_k\right) &= \left\{(a,a)\colon a\in\mathcal{A}^t,\quad \diam\left(K_a\right)\leq \varepsilon_k\right\},\\
	A_1\left(x,t,\varepsilon_k\right) &= \left\{(a,b)\in\mathcal{A}^t\times\mathcal{A}^t\colon a\neq b,\quad \diam\left(K_a\cup K_b\right)\leq \varepsilon_k\right\}.
	\end{split}
	\end{equation*}
	Using \eqref{eq:Prop:rdet:diam} we have that for every $a\in\mathcal{A}^t$,
	\begin{equation}\label{eq:Prop:det:every}
	\diam\left(K_a\right)\leq 2r^{-t}<\varepsilon_k,
	\end{equation}
	and so
	\begin{equation}\label{eq:Prop:det:A0}
	\#A_0\left(x,t,\varepsilon_k\right) = 2^t.
	\end{equation}
	For distinct $a,b\in\mathcal{A}^t$, exactly one of the following four cases is true:
	\begin{enumerate}[label=(\roman*), ref=(\roman*)]
		\item\label{it:Prop:0} \emph{$K_a\cup K_b\subseteq K_c$ for some $c=0c_1c_2\dots c_{t-2}\in\mathcal{A}^{t-1}$:} then \eqref{eq:Prop:rdet:diam} yields that $\diam\left(K_a\cup K_b\right) = \diam\left(K_c\right) = r^{-(t-1)} = \varepsilon_k$;
		\item\label{it:Prop:1} \emph{$K_a\cup K_b\subseteq K_c$ for some $c=1c_1c_2\dots c_{t-2}\in\mathcal{A}^{t-1}$:} then \eqref{eq:Prop:rdet:diam} yields that $\diam\left(K_a\cup K_b\right) = \diam\left(K_c\right) = 2r^{-(t-1)} > \varepsilon_k$;
		\item\label{it:Prop:01} \emph{either $K_a\cup K_b\subseteq K_0$ or $K_a\cup K_b\subseteq K_1$, and any from the previous cases is true:} then~\eqref{eq:Prop:rdet:dist} and the fact that $r\geq5$ yield that $\diam\left(K_a\cup K_b\right)\geq \dist\left(K_a, K_b\right)\geq (r-2)r^{-(t-1)}>\varepsilon_k$;
		\item\label{it:Prop:other} \emph{$K_a\subseteq K_0$ and $K_b\subseteq K_1$, or vice versa:} then~\eqref{eq:Prop:rdet:dist01} and the fact that $r\geq5$ yield that $\diam\left(K_a\cup K_b\right)\geq \dist\left(K_0, K_1\right) = 1-3/r>1/r=\varepsilon_1\geq\varepsilon_k$.
	\end{enumerate}
	Therefore, for distinct $a,b\in\mathcal{A}^t$, $\diam\left(K_a\cup K_b\right)\leq \varepsilon_k$ if and only if \ref{it:Prop:0} is true. There are $2^{t-2}$ pairwise distinct $c=0c_1c_2\dots c_{t-2}\in\mathcal{A}^{t-1}$ and for every such $c$ there are two intervals $K_a,K_b\in Q_t$ such that they both are subintervals of $K_c$. Hence,
	\begin{equation}\label{eq:Prop:det:A1}
	\#A_1\left(x,t,\varepsilon_k\right) = 2\cdot2^{t-2}.
	\end{equation}
	From~\eqref{eq:Prop:rdet:N1}, \eqref{eq:Prop:det:A0} and~\eqref{eq:Prop:det:A1} we have
	\begin{equation*}
	\#N_1^\circ\left(x,t,\varepsilon_k\right) = 3\cdot 2^{t-1} = 3\cdot 2^k.
	\end{equation*}
	
	\smallskip
	\emph{Step~3. Proof of the fact that $N_m^\circ\left(x,t,\varepsilon_k\right)= 2^{k+1}$ for $t=k+1$.}
	\smallskip
	
	Recall that $N_m^\circ\left(x,t,\varepsilon_k\right) = \left\{(a,b)\in \mathcal{A}^t\times\mathcal{A}^t\colon \diam_m(K_a, K_b) \le \varepsilon_k\right\}$. Thus, $$N_m^\circ\left(x,t,\varepsilon_k\right)\subseteq N_1^\circ\left(x,t,\varepsilon_k\right),$$ and using~\eqref{eq:Prop:rdet:N1},
	$$N_m^\circ\left(x,t,\varepsilon_k\right)\subseteq A_0\left(x,t,\varepsilon_k\right)\sqcup A_1\left(x,t,\varepsilon_k\right).$$
	Trivially, $A_0\left(x,t,\varepsilon_k\right)\subseteq N_m^\circ\left(x,t,\varepsilon_k\right)$ since~\eqref{eq:Prop:det:every} holds for every $a\in\mathcal{A}^t$.
	Moreover, we show that $A_1\left(x,t,\varepsilon_k\right)\cap N_m^\circ\left(x,t,\varepsilon_k\right) = \emptyset$, and hence using~\eqref{eq:Prop:det:A0}, $$\#N_m^\circ\left(x,t,\varepsilon_k\right) = \#A_0\left(x,t,\varepsilon_k\right) = 2^t = 2^{k+1}.$$ To see that, let $(a,b)\in A_1\left(x,t,\varepsilon_k\right)$. The proof of Step~2 implies that~\ref{it:Prop:0} is true. Thus,~\eqref{eq:Prop:rdet:fKa} yields that~\ref{it:Prop:1} is true for $a+1, b+1$, and so (again from the proof of~Step~2) $(a+1,b+1)\not\in A_1\left(x,t,\varepsilon_k\right)$. Therefore, $\varepsilon_k< \diam\left(K_{a+1}\cup K_{b+1}\right)\leq \diam_m\left(K_a,K_b\right)$, i.e., $(a,b)\not\in N_m^\circ\left(x,t,\varepsilon_k\right)$.

	\smallskip
	\emph{Step~4. Computation of $N_1^\circ\left(x,t,\varepsilon_k\right)$ and $N_m^\circ\left(x,t,\varepsilon_k\right)$ for $t\geq k+1$.}
	\smallskip
	
	Now, fix an integer $t\geq k+1$. Then, $(a,b)\in N_1^\circ\left(x,t,\varepsilon_k\right)$ if and only if there are -- not necessarily distinct -- $c,d\in\mathcal{A}^{k+1}$ such that $K_a\subseteq K_c, K_b\subseteq K_d$ and $(c,d)\in N_1^\circ\left(x,k+1,\varepsilon_k\right)$. Equivalently, the previous conditions are true if and only if there are $c,d\in\mathcal{A}^{k+1}$ such that $a=ca_{k+1}a_{k+2}\dots a_{t-1}$, $b=db_{k+1}b_{k+2}\dots b_{t-1}$ for any $a_i, b_i\in\{0,1\}$, $k+1\leq i < t$, and $(c,d)\in N_1^\circ\left(x,k+1,\varepsilon_k\right)$. Therefore, for such $(c,d)$, there are $2^{2(t-k-1)}$ pairs $(a,b)$ for which the above is true. One can get an~analogous result for $(a,b)\in N_m^\circ\left(x,t,\varepsilon_k\right)$. Hence, Steps~2 and~3 yield
		\begin{equation}\label{eq:Prop:det:values}
		\#N_1^\circ\left(x,t,\varepsilon_k\right) = 4^{t-(k+1)}\cdot 3\cdot 2^k\quad \text{ and }\quad \#N_m^\circ\left(x,t,\varepsilon_k\right) = 4^{t-(k+1)}\cdot 2^{k+1}.
		\end{equation}

	\smallskip
	\emph{Step~5. Proof that	 
	$\liminf\nolimits_{\varepsilon\to0}\rdet_m(x,\varepsilon)$ and $\liminf\nolimits_{\varepsilon\to0}\lim\nolimits_{n\to\infty}\DET_m(x,n,\varepsilon)$are smaller than one.}
	\smallskip
	
	By~Theorem~\ref{Thm:solenoidal}, the asymptotic correlation sum exists for all $m\in\mathbb{N}$ and $\varepsilon>0$, and
	$$c_m(x,\varepsilon) = \lim\limits_{t\to\infty}\frac{\#N_m^\circ(x,t,\varepsilon)}{2^{2t}},$$
	where, for every $t\in\mathbb{N}$, one can choose $p_t=2^t$ in the assertion of~the~theorem since $f$ has zero entropy and $\omega_f(x)$ is solenoidal (see \cite[Theorems~5.4 and~4.1(d)]{Blo95} and~\cite{Smi86}). Hence,
	\begin{equation*}
	\rdet_m(x,\varepsilon) = \frac{c_m(x,\varepsilon)}{c_1(x,\varepsilon)} = \lim\limits_{t\to\infty}\frac{\#N_m^\circ(x,t,\varepsilon)}{\#N_1^\circ(x,t,\varepsilon)}
	\end{equation*}
	for all $m\in\mathbb{N}$ and $\varepsilon>0$. 
	Thus, for $m\in\mathbb{N}\setminus\{1\}$, \eqref{eq:Prop:det:values} yields
	\begin{equation*}
	\liminf\limits_{\varepsilon\to0}\rdet_m(x,\varepsilon) \leq \lim\limits_{k\to\infty}\rdet_m\left(x,\varepsilon_k\right) = \frac{2}{3} <1.
	\end{equation*}
	Further, by~\eqref{eq:Intro_det},
	\begin{equation*}
	\begin{split}
		\liminf\limits_{\varepsilon\to0}\lim\limits_{n\to\infty}\DET_m(x,n,\varepsilon) & = \liminf\limits_{\varepsilon\to0} \left( m\cdot\rdet_m(x,\varepsilon) - (m-1)\cdot\rdet_{m+1}(x,\varepsilon) \right)\\
		& \leq \lim\limits_{k\to\infty} \left( m\cdot\rdet_m\left(x,\varepsilon_k\right) - (m-1)\cdot\rdet_{m+1}\left(x,\varepsilon_k\right) \right)\\
		& = \frac23 <1.
	\end{split}
	\end{equation*}
\end{proof}

\medskip
\noindent\emph{Acknowledgements.}
The author is very obliged to Vladim\'{i}r \v{S}pitalsk\'{y} for careful reading of the manuscript and giving substantive feedback. Furthermore, the author thanks Mat\'{u}\v{s} Dirb\'{a}k for useful comments.
This work was supported by VEGA grant 1/0158/20.


\end{document}